\documentclass[12pt,reqno]{amsart}

\usepackage{amsmath, amsthm, amsopn, amssymb, enumerate}

\usepackage{amsmath , amssymb, latexsym }
\usepackage{graphics}
\usepackage{graphicx}
\usepackage{verbatim} 

\usepackage[mathscr]{eucal}

\theoremstyle{plain}
\newtheorem{corollary}{Corollary}[section]
\newtheorem*{corollary*}{Corollary}
\newtheorem{proposition}{Proposition}[section]
\newtheorem*{proposition*}{Proposition}
\newtheorem{theorem}{Theorem}[section]
\newtheorem*{theorem*}{Theorem}
\newtheorem{lemma}{Lemma}[section]
\newtheorem*{lemma*}{Lemma}
\newtheorem{conjecture}{Conjecture}
\newtheorem{question}{Question}
\newtheorem*{claim}{Claim}
\newtheorem*{CD}{The Cauchy-Davenport Thorem}
\newtheorem*{Ruzsa}{Ruzsa's triangle inequality}
\theoremstyle{definition}
\newtheorem{definition}{Definition}[section]
\newtheorem*{definition*}{Definition}
\newtheorem*{akn}{Aknowledgement}

\theoremstyle{remark}

\newtheorem*{obs*}{Observation}

\newtheorem{remark}{Remark}

\newcommand{\Z}[1]{\mathbb{Z}_{#1}}

\def\ex{\mathbf{ex}}
\def\ZZ{\mathbb{Z}}
\def\NN{\mathbb{N}}

\title{Sums of Dilates in $\Z{p}$}

\date{\today}

\author{Gonzalo Fiz Pontiveros}
\address{IMPA\\ 
Estrada Dona Castorina 110\\ 
Jardim Bot\^anico\\ 
Rio de Janeiro\\ 
RJ\\
Brasil} 
\email{gf232@cam.ac.uk}

\begin{document}
\begin{abstract}
We consider the problem of sums of dilates in groups of prime order. We show that given $A\subset \Z{p}$ of sufficiently small density then $$\big| \lambda_{1}A+\lambda_{2}A+\ldots+ \lambda_{k}A \big| \,\ge\,\bigg(\sum_{i}|\lambda_{i}|\bigg)|A|- o(|A|),$$
whereas on the other hand, for any $\epsilon>0$, we construct subsets of density $1/2-\epsilon$ such that $|A+\lambda A|\leq (1-\delta)p$, showing that there is a very different behaviour for subsets of large density. 
\end{abstract}
\maketitle

\section{Introduction}
Combinatorial Number Theory is an area of mathematics which deals with the additive and multiplicative structure of sets, and encompasses techniques from harmonic analysis, ergodic theory, graph theory and number theory (see Tao and Vu~\cite{TaoVu}, for example). A central notion in this field is that of a \emph{sumset}: given an abelian group $G$, a finite subset $A \subset G$ and integers $\lambda_1,\ldots,\lambda_k \in \Z{}$, consider the set
$$X = \lambda_{1}A + \lambda_{2}A + \ldots + \lambda_{k}A,$$
where $A + B = \big\{ a+b : a\in A,\; b\in B \big\}$ and $\lambda A = \big\{ \lambda a : a \in A \big\}$. Finding lower bounds for the size of $X$ (in term of $|A|$) is a basic and central question in the area.

For the group $G = \Z{}$, the situation is fairly well understood. It is easy to prove that $|A+A| \ge 2|A| - 1$, with equality if and only if $A$ is an arithmetic progression. Plagne and Hamidoune \cite{Hamidoune:2002ie} gave  the lower bound $|A+\lambda A|\geq 3|A|-2$ for $\lambda \notin \{-1,0,1\}$. Later on, for $|\lambda|\geq3$, Nathanson~\cite{Nathanson:2008tu} showed that  $|A+\lambda A|\geq \frac{7}{2}|A|- \frac{5}{2}$. 
The case $\lambda=3$ was completely solved in~\cite{Bukh:2008p565} and~\cite{Cilleruelo:2010tn}; in the latter, Cilleruelo, Silva and Vinuesa also characterise all of the extremal sets.

	In the general case, the great breakthrough is due to Bukh, who gave an asymptotically sharp bound in~\cite{Bukh:2008p565}.

\begin{theorem}[Bukh, 2008]\label{Bukh}
For any $k$ coprime integers $\lambda_{1},\lambda_{2},\ldots,\lambda_{k}$ and any finite set $A\subset \ZZ$, we have
\begin{equation}
\big| \lambda_{1} A+\lambda_{2}A+\ldots+ \lambda_{k}A \big| \,\ge\, \bigg(\sum_{i}|\lambda_{i}|\bigg)|A|- o(|A|)
\end{equation}
where the error term $o(|A|)$ depends only on $\lambda_{1},\ldots,\lambda_{k}$.
\end{theorem}

However, in the finite field setting $G = \Z{p}$ the problem turns out to be much harder, and almost nothing is known. The only exception is the case $k = 2$ and $\lambda_1 = \lambda_2 = 1$, for which we have the celebrated  \emph{Cauchy-Davenport Theorem}~\cite{Cauchy:1813uva,Davenport:1935toa}:

\begin{CD}
Let $p$ be prime, and let $A, B \subset \Z{p}$ be non-empty. Then
\begin{equation*}
|A+B| \,\ge\, \min\big\{ |A|+|B|-1,p \big\}.
\end{equation*}
\end{CD}

\noindent Thus, in particular, we have the bound $$|A+A|\ge 2|A|-1\;\; \textit{for every}\;\; |A|\leq\frac{p-1}{2} .$$ 
Furthermore, Vosper~\cite{Vosper:1956tp} showed that, as in the case $G=\Z{}$, equality holds if and only if both $A,B$ are arithmetic progressions with the same common difference.

In the light of the two theorems above, it seems natural to conjecture that if $p$ is prime, $\lambda_1, \lambda_2 \in \ZZ$ and $A \subseteq \Z{p}$ then
\begin{equation}
\label{naive}
|\lambda_1A+ \lambda_2A| \ge \min\Big\{ (|\lambda_1|+|\lambda_2|)|A| - O(1),\; p \Big\}.
\end{equation}

However, we have shown that~\eqref{naive} fails in a very strong sense. We prove the following result.

\begin{theorem}
\label{large}
For every $\lambda \in \Z{}$ and $\epsilon>0$, there exists $\delta=\delta(\lambda,\epsilon)>0$ such that the following holds. If $p$ is a sufficiently large prime, then there exists a set $A\subset \Z{p}$ with  $|A|\geq (1/2-\epsilon)p$ such that $|A+\lambda A|\leq (1-\delta)p$. 
\end{theorem}

Although the result is surprising, the proof of Theorem~\ref{large} is relatively straightforward. The key step is to transfer the problem to the continuous setting, where one can apply a known result on the dynamics of measure preserving maps in the circle. In the general setting of~\cite{Bukh:2008p565} we  make the following conjecture:

\begin{conjecture}
\label{con1}
For every $\lambda_1,\ldots,\lambda_k \in \Z{}$ and $\epsilon > 0$, there exists $\delta > 0$ such that the following holds. For every prime $p$, there exists a set $A \subset \Z{p}$ with $|A| \ge (1/k - \epsilon)p$ such that 
$$\big| \lambda_{1}A + \lambda_{2}A + \ldots + \lambda_{k}A \big| \, \le \, (1 - \delta)p.$$
\end{conjecture}

On the other hand, if $A$ is much smaller then we show that the bound~\eqref{naive} \emph{does} hold:

\begin{theorem}\label{smallA}
For every coprime $\lambda_1,\ldots,\lambda_k \in \Z{}$, there exists a constant $\alpha > 0$ such that 
$$\big| \lambda_{1}A+\lambda_{2}A+\ldots+ \lambda_{k}A \big| \,\ge\,\bigg(\sum_{i}|\lambda_{i}|\bigg)|A|- o(|A|)$$ 
for every sufficiently large prime $p$, and every $A \subset \Z{p}$ with $|A| \le \alpha p$.
\end{theorem} 
 Very recently,  Plagne~\cite{Plagne:2011eu} has given bounds along the same spirit for sumsets of the form $A+\lambda A$. Namely, it was shown that there exists a function $f_\lambda$ and a constant $w(\lambda)$ such that
\begin{theorem}[Plagne, 2011]
 \begin{equation}
 \label{Plagne}
    |A+ \lambda A|\geq \min\Big(f_\lambda(\alpha)|A|-w(\lambda), p\Big),
    \end{equation}
where $|A|=\alpha p$ and $f_\lambda(\alpha)$ is defined to be the maximum of $2$ and the unique solution to the equation
$$(\lambda+1)\sin{\Big(\frac{\pi}{\lambda+1}\Big)}(1-\alpha x)=x^{3/2}\sin{\Big(\frac{\pi}{x}\Big)}.$$
\end{theorem}

The range of densities $\alpha$ for which the function $f_\lambda$ yields non-trivial bounds (i.e., $f_\lambda(\alpha)>2$) is much bigger than the one obtained in Theorem ~\ref{smallA} and tends to $\frac{1}{2}-\frac{\sqrt{2}}{\pi}\approx 0.0498$ as $|\lambda|\to \infty$. However, the value of $f_\lambda(\alpha)$ is far from the truth when looking at very small $\alpha$, indeed $f_\lambda(0)< 2.16$ for all $\lambda \in \Z{}$. So the bound for large $\lambda$ is essentially  
$$|A+\lambda A|\geq \min\Big( 2.16 |A|, 0.0996p\Big),$$
whereas our method yields
$$|A+\lambda A|\geq \min\Big( (\lambda+1)|A|,\epsilon(\lambda)p\Big).$$
Hence the bounds are in a sense complementary.

\section{Small densities}
The strategy to deal with sets of small density is to show that they in fact behave like subsets of the integers with respect to addition and thus we may transfer the relevant bounds. This kind of rectification technique was first introduced by Freiman~\cite{Freiman} in the proof of his well known theorem and is now a standard tool to transfer results from the integer setting. 

Suppose that the set $A$ is contained in the interval $\Big\{0,\ldots, \big\lfloor \frac{p}{M}\big \rfloor \Big\}$ where $M=\sum|\lambda_i|$. It is clear that when looking at sums of the form $\lambda_1a_{1}+\lambda_2a_{2}+\ldots \lambda_k a_k$ for $a_1, a_2, \ldots, a_k \in A$ there is never any `wrap around' and hence we might as well think of them as happening in $\Z{}$. We make this statement more precise:

\begin{definition}
The {\it diameter} of a subset $A\subseteq \Z{p}$, denoted $l(A)$,  is defined as the
smallest integer $l$ for which there exists some $x,d \in \Z{p}$ such that $A \subseteq\{x,x + d,...,x + (l-1)d\}$. In words, the length of the shortest arithmetic progression containing $A$.
 \end{definition}
\begin{definition}
Let $A$ and $B$ be finite subsets of abelian groups (not necessarily equal). We say that $A$ and $B$ are $m$-\emph{Freiman isomorphic} if there exists a bijection $f:A\to B$ such that 
\begin{eqnarray*}
f(a_1)+f(a_2)\ldots+f(a_m)&=& f(a_1')+f(a_2')\ldots+f(a_m')\\
&\Updownarrow&\\
a_1+\ldots+a_m&=&a_1'+\ldots+a_m'.
\end{eqnarray*}
\end{definition}

Formally speaking, if $l(A)< \frac{1}{M}p$ where $M=\sum_1^k|\lambda_i|$, then there exists $\tilde A \subset \Z{}$ such that $\tilde A$ is $M$-Freiman isomorphic to $A$. In particular it follows that, $|\lambda_1\tilde A + \lambda_2\tilde A+\ldots +\lambda_k\tilde A|=|\lambda_1A +\lambda_2A+\ldots \lambda_kA|$ and therefore we are free to apply any known bounds in the integer setting. 
\begin{corollary}

\label{diameter}
Let $\lambda_1, \lambda_2\ldots \lambda_k$ be relatively prime integers and $A\subset \Z{p}$ be a subset with diameter $l(A)<\frac{1}{M}p$ where $M=\sum_1^k|\lambda_i|$. Then 
\begin{equation}
|\lambda_1A+\lambda_2 A+\ldots +\lambda_kA|\geq \big(|\lambda_1|+|\lambda_2|+\ldots +|\lambda_k|\big)|A| - o(|A|),
\end{equation}
where the $o(|A|)$ is the error term given by Theorem~\ref{Bukh}. 
\end{corollary}
\begin{remark}
In the case $k=2$, $\lambda_1\in \{1,2\}$ and $\lambda_2\in\{\pm q\}$ for $q$ a prime,  very precise error terms of the form  $g(\lambda_1,\lambda_2)$ are given in~\cite{Cilleruelo:2009cf} and  \cite{Hamidoune:2011wm} for the integer case and hence are also transferable to $\Z{p}$. 
\end{remark}
A priori, we can only guarantee that $l(A)\leq \frac{1}{k}p$ for sets with $|A|<\frac{\log p}{\log k}$ (using the pigeon hole principle) and this bound is best possible up to a constant: for example, a random subset of $\Z{p}$ of size $4\log p$  has diameter larger than $p/2$ with high probability. 

However, if we are given the additional information that the sumset  $A+A$ is small,  then it is possible to bound well the diameter of the set $A$. Freiman~\cite{Freiman} showed that for all $C>0$ there exists $\epsilon(C)>0$ such that 
if $|A+A|\leq C |A|$ and $|A|< \epsilon(C)p$ then the set $A$ is Freiman isomorphic to a subset of the integers. More recently, Green  and Ruzsa \cite{Green:2006p564} showed the same result with a different approach that yielded much better bounds. We will use their formulation:

\begin{theorem}[Green-Ruzsa, 2006]
\label{Green-Ruzsa}
Let $A$ be a subset of $\Z{p}$, $p$ a prime, with $|A|=\alpha N$ and assume that $\min\{|A+A|,|A-A|\}\leq K|A|$. If $\alpha\leq (16K)^{-12K^{2}}$, then the diameter of $A$ is at most 
\begin{equation}
12\alpha^{1/4}\sqrt{\log{(1/\alpha)}}N.
\end{equation}
\end{theorem}
Equipped with this tool, it is now a simple task to deduce a bound for sets of small density by transferring the known $\Z{}$ bounds: 
\begin{corollary}
\label{smallsets}
For every $\bar\lambda=(\lambda_1,\lambda_2,\ldots,\lambda_k)$, there exits some $\alpha(\bar\lambda)>0$ such that for all $A\subset \Z{p}$ with $|A|\leq \alpha p$,
\begin{equation*}
|\lambda_1A+ \lambda_2 A+\ldots+\lambda_k A|\geq \Big(|\lambda_1|+|\lambda_2|+\ldots+|\lambda_k|\Big)|A|- o(|A|).
\end{equation*}
\end{corollary}
\begin{proof}
Let $M=\sum|\lambda_i|$. We may assume that $|\lambda_1A+ \lambda_2A+\ldots+\lambda_kA|\leq  M|A|$ as otherwise we are done. We  make use of the following elementary inequality due to Ruzsa~\cite{Ruzsa:1996vz}:
\begin{Ruzsa}
Given any $B,C,D$ finite subsets of some abelian group $G$. Then, 
\begin{equation}
|B+D|\leq \frac{|B+C||C+D|}{|C|}.
\end{equation}
\end{Ruzsa}
\noindent Setting $B=D=\lambda_1A$ and $C=\lambda_2A+\ldots+\lambda_k A$ we have that 
\begin{equation}
|A+A|=|\lambda_1A+\lambda_1A|\leq \frac{|\lambda_1A+\ldots+\lambda_k A|^{2}}{|\lambda_2A+\ldots+\lambda_k A|}\leq M^2|A|
\end{equation}
Now pick $\alpha>0$ such that $12\alpha^{1/4}\sqrt{\log{(1/\alpha)}}<M^{-1}$
and $\alpha<(16M)^{-12M^{2}}$. By the Green-Rusza Theorem, any set $A\subset \Z{p}$  with $|A|\leq \alpha|A|$
and $|A+A|\leq M|A|$ will have diameter at most $M^{-1}p$ and hence it follows from Corollary~\ref{diameter} that
\begin{equation*}
|\lambda_1A+\ldots+\lambda_kA|\geq \Big(|\lambda_1|+\ldots+|\lambda_k|\Big)|A|- o(|A|),
\end{equation*}
as required.
\end{proof}

\section{Large densities}
We begin with a simple lemma that already shows that the behaviour of sets of large density is indeed quite different from the integers. Indeed, if \eqref{naive} were true, then any set with density greater than $\frac{1}{|\lambda_1|+|\lambda_2|}$ ought to satisfy $\lambda_1A+\lambda_2A = \Z{p}$. Instead we have, 

\begin{lemma}
\label{big}
For any prime $p$, we may find a set $A\subset \Z{p}$ such that $|A|=\frac{1}{2}p-o(p)$ and $0 \notin \lambda_1A+\lambda_2A$ 

\end{lemma}
\begin{proof}
Define a digraph $\Gamma$ with on the vertex set  $\Z{p}$ where $x\to y$ if and only if $\lambda_1x+\lambda_2y=0$. We are done if we find a large independent set in $\Gamma$. Now for any $x \in \Z{p}$ we have that $d_{-}(x)=d_{+}(x)=1$ and hence $\Gamma$ is the disjoint union of directed cycles. 

Suppose we have a cycle $x_{1}\to x_{2}\to \ldots \to x_{k}\to x_{1}$. By definition of $\Gamma$ we have that $\lambda_2 x_{i+1}=(-\lambda_1)x_{i} \mod{p}$ for all $1\geq i \geq k-1$. Hence, $k$ can only be the smallest positive integer for which we have that 
$$ (-\lambda_1)^{k}= \lambda_2^k \mod{p}$$
Certainly the above can only happen if $k\geq \lfloor\log_{\max\{|\lambda_1|,|\lambda_2|\}}{p}\rfloor-1$. Thus all cycles in $\Gamma$ have the same length $$k= \Omega(\log p)$$

This is good news since it is simple to find independent sets in cycles by picking alternative vertices, that is as many as $\lfloor k/2 \rfloor$ vertices in each cycle.
In this way we obtain an independent set $A$ of density $\frac{\lfloor k/2 \rfloor}{l}\geq\frac{1}{2} -O(\frac{1}{\log p})$ and therefore
$$|A|\geq \frac{1}{2}p - O\Big(\frac{p}{\log{p}}\Big),$$
as claimed.
\end{proof}
Note that, on the other hand, the Cauchy-Davenport Theorem gives us the lower bound
$$|\lambda_1A+\lambda_2 A|\geq \min\{ 2|A|-1,p\}$$
and thus any set of density above $1/2$ must satisfy that $\lambda_1A+ \lambda_2A=\Z{p}$.  
 
 One might wonder if this is some kind of exception and if instead we demand that $|\lambda_1A+\lambda_2A|\leq p-2$ then we have a non trivial upper bound on the density of $A$, meaning that it is at most $1/2-\epsilon$ for some $\epsilon>0$. 
We shall show that this is not the case; our strategy is to transfer the problem into a continuous setting. The following proposition makes this transference  explicit:

\begin{lemma}
\label{transfer}
Let $\mathbb{T}$ denote the unit circle, equipped with the Lebesgue measure $\mu$ and $\lambda_1, \lambda_2 \in \Z{}$. For any $\alpha \in [0,1]$ the following are equivalent:
\begin{itemize}
\item[(i)] Given any $\epsilon_{1}>0$, there exists a measurable set $A\subset \mathbb{T}$ with $\mu(A)\geq \alpha-\epsilon_{1}$ such that $\lambda_1A\cap \lambda_2A=\emptyset$.
\item[(ii)] Given any $\epsilon_{2}>0$, there exists $p_{0}$ and $\delta>0$ such that for all primes $p\geq p_{0}$ and any $t\le \delta p$, there exists $A\subset \Z{p}$ with $|A|\geq (\alpha -\epsilon_{2})p$, and satisfying $\{-t, -t+1,\ldots,0,1,\ldots, t\}\notin \lambda_1A-\lambda_2A$.
\end{itemize}

\end{lemma}
\begin{proof}
We begin with the easier implication $(ii)\Rightarrow(i):$ \\

By assumption, given any $\epsilon>0$ we may find a prime $p$ and a set $A\subset \Z{p}$ such that $|A|\geq (\alpha-\epsilon)p$ and such that the interval $\big\{ -\lceil \frac{|\lambda_1|+|\lambda_2|}{2} \rceil,\ldots, \lceil \frac{|\lambda_1|+|\lambda_2|}{2} \rceil \big\}\ \notin \lambda_1A-\lambda_2A$. We embed $\Z{p}$ in the natural way, namely via the homomorphism $\chi_{p}: x\mapsto e^{\frac{2\pi ix}{p}}$ and define

\begin{equation}
\tilde{A}=\Bigg\{ e^{i\theta}: \exists\; x\in A\;\text{s.t}\quad \bigg|\bigg|\theta - \frac{2\pi x}{p}\bigg|\bigg|<\frac{\pi}{p}\Bigg\}
\end{equation}

In words, $\tilde{A}$ is the disjoint union of open intervals of length $\frac{2\pi}{p}$, centred  at $\chi_{p}(x)$ for $x \in A$. 

It is clear that $\mu(\tilde{A})=|A|(1/p)\geq \alpha-\epsilon$. It remains to show that $\lambda_1\tilde{A}\cap\lambda_2\tilde{A}=\emptyset$.  Suppose that $e^{i\theta}\in \lambda_1\tilde{A}\cap\lambda_2\tilde{A}$, then there exist $x, y \in A$ such that

\begin{eqnarray*}
\bigg|\bigg|\theta - \frac{2\pi (\lambda_1x)}{p}\bigg|\bigg|<|\lambda_1|\frac{\pi}{p}\quad\text{and}\quad \bigg|\bigg|\theta - \frac{2\pi (\lambda_2 y)}{p}\bigg|\bigg|<|\lambda_2|\frac{\pi}{p} .
\end{eqnarray*}
Thus, by the triangle inequality,
\begin{eqnarray*}
\bigg|\bigg|\frac{2\pi}{p}(\lambda_1x-\lambda_2 y)\bigg|\bigg|<(|\lambda_1|+|\lambda_2|)\frac{2\pi}{p}
\end{eqnarray*}
but since $x$ and $y$ are integers this can only happen if there is a solution in $A$ to the equation $\lambda_1x-\lambda_2 y= r \mod p$,  for some $r \in \big\{ -\lceil \frac{|\lambda_1|+|\lambda_2|}{2} \rceil,\ldots, \lceil \frac{|\lambda_1|+|\lambda_2|}{2} \rceil \big\}$. This is a contradiction to our choice of $A$ and thus $\lambda_1\tilde A\cap\lambda_2 \tilde A$ must be empty. \\

\noindent $(i)\Rightarrow(ii):$\\

Suppose that there exists  a measurable set $A$ of $\mathbb{T}$ with $\mu(A)\geq \alpha-\epsilon$ such that $\lambda_1A\cap \lambda_2 A =\emptyset$. 
Since $A$ is measurable, we may find a $A'=\bigcup_{i=1}^{K} [a_{i},b_{i}]$ such that  $\mu(A\setminus A')+\mu(A'\setminus A)< \epsilon$.
Certainly $\mu(A')\geq \alpha-2\epsilon$ and furthermore 
\begin{equation*}
\begin{split}
\mu(\lambda_1A'\cap \lambda_2A')&\leq \mu\big(\lambda_1(A'\setminus A)\cap \lambda_2A'\big)+\mu\big(\lambda_1A'\cap \lambda_2(A'\setminus A)\big)\\
&\leq |\lambda_1|\mu\big(A'\setminus A\big)+ |\lambda_2|\big(A'\setminus A\big)<\big(|\lambda_1|+|\lambda_2|\big)\epsilon.\\
\end{split}
\end{equation*}

\noindent Now we `discretise' the circle in the obvious way:\\
\begin{equation*}
A_{p}=\{x:  x\in \{0,\ldots,p-1\}, x/p \in A'\} \\
\end{equation*}
and view it as a subset of $\mathbb{Z}_{p}$. As $A'$ is the union of $K$ intervals, we also have the bounds
\begin{equation}
\label{discretise}
|A_{p}|-K\leq\mu(A')p\leq |A_{p}|+K\\
\end{equation}

\noindent Next, we claim that the set $A_{p}$ is close to satisfying the desired property. 

\begin{claim}
\begin{equation*}
\left|\bigcup\limits_{-t}^{t}\lambda_1A_{p}\cap(\lambda_2A_{p}+i)\right| \leq \big(|\lambda_1|+|\lambda_2|\big)\epsilon p + tK^{2} + tK
\end{equation*}
\end{claim}

\noindent First of all, note that $\lambda_1A'$ and $\lambda_2 A'$ are also the union of at most $K$ intervals, and hence the set $\lambda_1A' \cap \lambda_2 A'$ may certainly be expressed as the union of at most $K^{2}$ intervals.

\noindent We also have the following inequality, 

\begin{eqnarray}
\mu\bigg(\lambda_2 A'+\Big[-\frac{t}{p},\frac{t}{p}\Big]\setminus \lambda_2 A'\bigg)\leq 2tK/p. 
\end{eqnarray}

\noindent Thus,
\begin{eqnarray*}
\mu(\lambda_1A'\cap\lambda_2 A' +\{-t/p.\ldots, t/p\})&\leq&\mu(\lambda_1A'\cap\lambda_2 A')+\mu(\lambda_2 A'+[-t/p,t/p]\setminus \lambda_2 A')\\
&<& \big(|\lambda_1|+|\lambda_2|\big)\epsilon +2tK/p
\end{eqnarray*}

\noindent Now, applying equation \eqref{discretise} to the set $\lambda_1A' \cap (\lambda_2 A'+\{-t/p \ldots, t/p\})$, which is the union of at most $2tK^{2}$ intervals: 

\begin{equation*}
\left|\lambda_1A_{p}\cap \big(\lambda_2 A_{p} +\{-t, \ldots, t\}\right|\leq \big(|\lambda_1|+|\lambda_2|\big)\epsilon p +2tK^{2} +2 tK
\end{equation*}
Hence, by deleting at most $\big(|\lambda_1|+|\lambda_2|\big)\epsilon p  +2tK^{2} + 2tK$ points from $A_{p}$ (namely those in $\lambda_1A_{p}\cap(\lambda_2A_{p}+\{-t, \ldots, t\})$, we can find a set $A$ in $\mathbb{Z}_{p}$ such that $0,\ldots, t \notin A+ \lambda A$ and $|A| \geq (\alpha-\big(|\lambda_1|+|\lambda_2|\big)\epsilon)p-O_{K}(\delta)p\geq(\alpha-2(|\lambda_1|+|\lambda_2|)\epsilon))p$, provided we have chosen $\delta$ sufficiently small.
\end{proof}

The task now is that of finding sets of large (Lebesgue) measure in $\mathbb{T}$ such that $\lambda_1A\cap\lambda_2A=\emptyset$. Firstly, note that we trivially have the upper bound $\mu(A)\leq 1/2$ since $\mu(\lambda_iA)\geq \mu(A)$. 

We address below the special case where $\lambda_1=1$ and begin by showing that in fact we cannot have that $\mu(A)=1/2$.

\begin{proposition}
Let $A$ be a subset of $\mathbb{T}$ such that $A\cap \lambda A=\emptyset$, then $\mu(A)<1/2$.
\end{proposition}

\begin{proof}
Suppose that there exists such a subset $A\subset \mathbb{T}$ with $\mu(A)=1/2$. Then $(-\lambda) A=\bar A=\mathbb{T}\setminus A$ (modulo a set of null measure) and, since multiplication by $-\lambda$ is a measure preserving transformation of the circle, it also follows that $\lambda^2A=A$. However, it is well known that multiplication by an integer is an ergodic transformation, that is to say, the only invariant subsets have either null or full measure. This is a contradiction. 
\end{proof}

 We will again consider the action of multiplication by $-\lambda$ as an ergodic, measure preserving transformation of $\mathbb{T}$. The following result was first proven by Rokhlin in \cite{Rokhlin:1963uv} . We won't need here the full generality of Rokhlin's Lemma as the transformations we are interested here are ergodic. We include here a proof in the setting of ergodic transformations of the circle, as it is very short and simple and it also contains the underlying ideas of the `quatitative version' we give in the following section. For the more general result, we recommend the reader to see \cite{Heinemann:2001p517} which presents a remarkably simple proof of Rohklin's Lemma, even when the transformations are not invertible.
 
 \begin{lemma}[Rokhlin's Lemma]
 \label{rokhlin}
 Let $\phi: \mathbb{T}\to\mathbb{T}$ be an ergodic, measure preserving, mesurable map. Given any $n\in \mathbb{N}\setminus\{0\}$ and any $\epsilon>0$, there exists a measurable set $B\subset \mathbb{T}$ such that 
 \begin{itemize}
\item[(i)] $\Big(\phi^{i}(B)\Big)_{0\leq i\leq n-1}\quad$ are pairwise disjoint.\\
 \item[(ii)] $\mu(B)\geq \frac{1}{n}-\epsilon$.
 \end{itemize}
\end{lemma}

\begin{proof}
Pick any measurable set $E\subset \mathbb{T}$ with $0<\mu(E)\leq \epsilon$, for example, the open ball centred at $0$ of radius $\epsilon/2$. We construct a family of sets as follows
\begin{equation}
E_{0}=E; \qquad E_{i+1}=\phi^{-1}(E_{i})\setminus \bigcup\limits_{0\leq j\leq i}E_{j}\\
\end{equation}

The sets $(E_{i})_{i\geq0}$ are, by construction, pairwise disjoint and for any $i \in \mathbb{N}$ we have the inclusion  
\begin{equation}
\label{chain}
\phi(E_{i+1})\subseteq E_{i}.
\end{equation}

\begin{claim}
\begin{equation}
\bigcup\limits_{i\geq0}E_{i}=\mathbb{T}\qquad \text{$\mu$-almost everywhere}\\
\end{equation}
\end{claim}

\noindent Set $F=\bigcup\limits_{i\geq0}E_{i}=\bigcup\limits_{i\geq0}\phi^{-i}(E_{0})$	and note that
\begin{equation*}
\phi^{-1}(F)=\bigcup\limits_{i\geq1}\phi^{-i}(E_{0})\subseteq F 
\end{equation*}

On the other hand, as $\phi$ is measure preserving, $\mu\big(\phi^{-1}(F)\big)=\mu(F) \Rightarrow \phi^{-1}(F)=F$. Thus, by ergodicity, either $\mu(F)=0$ or $\mu(F)=1$. But $F$ contains $E$, which has a strictly positive measure, so it must be the case that $\mu(F)=1$.

\noindent Let $B=\bigcup_{j\geq1}E_{jn}$, we claim that $B$ meets the requirements of Rokhlin's Lemma:

\begin{itemize}
\item[(i)] For convenience, let $J_{r}$ denote the set $\{ k\in \mathbb{N}^{*}: k=r\mod n\}$. From \eqref{chain} we have that for each $i=0,1,\dots, n-1$
\begin{equation}
\phi^{i}(B)=\phi^{i}\left ( \bigcup\limits_{j\in J_{0}} E_{j}\right)=\bigcup\limits_{j\in J_{0}}\phi^{i}(E_{j})\subseteq \bigcup\limits_{j\in J_{0}}E_{j-i}=\bigcup\limits_{j\in J_{n-i}}E_{j}
\end{equation}
and therefore the sets $\Big(\phi^{i}(B)\Big)_{i=0}^{n-1}$ are pairwise disjoint.
\item[(ii)] We have the inclusion
\begin{equation}
\bigcup\limits_{i=0}^{n-1}\phi^{-i}(A)\supseteq\bigcup\limits_{i\geq0}E_{i}\setminus\bigcup\limits_{i\geq0}^{n-1}E_{i} =\mathbb{T}\setminus \bigcup\limits_{i\geq0}^{n-1}E_{i}
\end{equation}
now recall we chose $E_{0}$ with $\mu(E_{0})\leq \epsilon$, hence
\begin{equation*}
\mu\left(\bigcup\limits_{i=0}^{n-1}\phi^{-i}(A)\right)=n\mu(B)\geq 1- n\epsilon
\end{equation*}
and therefore $\mu(B)\geq \frac{1}{n}-\epsilon$.
\end{itemize}
\end{proof}

\noindent As an immediate consequence of Rokhlin's Lemma we have: 
\begin{proposition}
For any given $\epsilon>0$, there exists a set  $A\subset \mathbb{T}$ such that $A\cap \lambda A=\emptyset$ and $\mu(A)>1/2-\epsilon$. 
\end{proposition}

Theorem~\ref{large} now easily follows: by the above,  for every $\lambda \in \Z{}$ and any $\epsilon>0$, there exists $A\subset\mathbb{T}$
such that $\mu(A)>1/2-\epsilon$. Applying Lemma~\ref{transfer} to the set $A$ we obtain $A'\subset \Z{p}$ (for some sufficiently large prime $p$) such that $(A'+\lambda A')\cap [0,\delta p]=\emptyset$ and hence in particular
$$|A'+\lambda A'|\leq (1-\delta)p,$$
as claimed.
\section{A Quantitative Bound}

 Since we are considering a very specific family of maps, it is possible to make the construction in  Rohklin's Lemma  very explicit. Indeed , one can avoid appealing to measure theory or limiting arguments, which has the added advantage of providing an explicit dependency between the constants involved. At a first  instance we will assume that $\lambda$ is a positive integer and later on we will show how one can remove this assumption.

\noindent Let $m\in \mathbb{N}$ and set $E_0=[0, \lambda^{-m})$, then  for each $i \in \NN$ define
\begin{eqnarray*}
E_i=\{x\in \mathbb{T}: \lambda^ix\in E_0\; \text{but}\; \lambda^j x \notin E_0\; \text{for all}\; j<i\} 
\end{eqnarray*}
 More combinatorially, we may think of  $E_{i}$ as the set of $x \in \mathbb{T}$ such that the first occurrence of $m$ consecutive zeros in the $\lambda$-ary expansion of $x$ appears at position $i$.
 
  Clearly, the sets $E_{i}$ are disjoint and satisfy the inclusion $\lambda\cdot E_{i+1}\subseteq E_{i}$. Furthermore, we also have that all but a small measure of the space:

 \begin{lemma}
 For any $n\in \mathbb{N}$, 
 $$\mu\left(\bigcap_{i=0}^{n-1}\bar{E_{i}}\right)\leq (1-\lambda^{-m})^{\frac{n}{m}}\leq e^{-\lambda^{-m}\frac{n}{m}}. $$
 \end{lemma}
\begin{proof}
We turn to the combinatorial interpretation of the sets $E_{i}$: to estimate $\mu\left(\bigcap_{i=0}^{n-1}\bar{E_{i}}\right)$, it will be sufficient to bound the number of sequences in $[\lambda]^{n+m}=\{0,1,\ldots, \lambda-1 \}^{n+m}$ without any $m$ consecutive zeros. In other words we are trying to bound above the probability that a uniformly chosen sequence in $[\lambda]^{n+m}$ contains $m$ consecutive zeroes. 

To do so, split the sequence into at least $k=\left \lfloor \frac{n+m}{m} \right \rfloor$ disjoint `blocks' of size $m$. The probability that each block is not a sequence of $m$ consecutive zeroes is $(1-\lambda^{-m})$. Hence, by independence, the probability that none of the blocks is a sequence of $m$ consecutive zeroes is at most 
$$(1-\lambda^{-m})^k\leq(1-\lambda^{-m})^{\frac{n}{m}}, $$ 	
as claimed.
 \end{proof}

\noindent Now for any $t\geq m$, let
$$A_{t}=E_{1}\sqcup E_{3}\ldots\sqcup E_{2t+1},$$
we have that $\lambda A_{t}\subseteq E_{0}\sqcup E_{2}\ldots\sqcup E_{2t}$ and thus $A_{t}\cap \lambda A_{t}=\emptyset$. Furthermore, it is easy to see that
 \begin{eqnarray*}
 \mu(A_{t})&\geq&\frac{1}{2}\left (\sum_1^{2t}\mu(E_i)\right)\geq \frac{1}{2}\left( 1-(1-\lambda^{-m})^{\frac{2t+1}{m}}-\lambda^{-m} \right)\\
 &=& \frac{1}{2}- \frac{1}{2}(1-\lambda^{-m})^{\frac{2t+1}{m}} -\frac{1}{2}\lambda^{-m}
 \end{eqnarray*}

Note that our set $A_t$ is already the union of disjoint intervals and therefore we do not need to approximate it as such in order to transfer it to $\Z{p}$. Furthermore, is easy too see that each $E_i$ is the union of at most $\lambda^{i}$ intervals and hence $A_t$ is the union of at most $\lambda^{4t}$ intervals. 

In order to remove the assumption that $\lambda$ is positive, note that it was only used when estimating 
$\mu\left(\bigcup_0^{n-1} E_i^{(\lambda)} \right )$. If $\lambda <0$, then one can obtain an estimate for the measure of the union by simply noting that that applying the transformation twice is the same as multiplication by $\lambda^2>0$. It is clear that,
$$\bigcup_{i=0}^{2(n-1)}E_i^{(\lambda)}\supset\bigcup_{i=0}^{n-1}E_i^{(\lambda^2)},$$
and thus we are well placed to use the previous estimate to bound the right hand side of the inclusion.

 Finally, setting $m=C_1\log\epsilon^{-1}$ and $t=C_2\epsilon^{-1}(\log\epsilon^{-1})^2$ where $C_1$ and $C_2$ are some sufficiently large constants, and transferring to $\Z{p}$ in the usual manner, we obtain the following quantitative version of Theorem $\ref{large}$:

\begin{theorem}
\label{large2}
Given any $\epsilon>0$ and $\lambda\neq 0  \in \Z{}$, for all sufficiently large primes $p$, there exists a set $A\subset \Z{p}$ with $|A|\geq 1/2-\epsilon$ such that $|A+\lambda A|\leq (1-\epsilon^{C\epsilon^{-1}})p$ for some universal constant $C$.  \\
\end{theorem}
\section{Comments and further questions}
Theorem~\ref{large} and Theorem~\ref{smallA} show that one cannot expect to obtain sharp bounds independent of the size of $A$. This motivates the following definition:
\begin{definition}
Given a prime $p$, a sequence of integers $\bar{\lambda} = (\lambda_1,\ldots,\lambda_k)$, and $\alpha \in [0,1]$, set
$$\ex\big( \Z{p},\bar{\lambda},\alpha \big)= \inf\left\{  \frac{| \lambda_{1}A+ \ldots + \lambda_{k}A |}{|A|} \,:\, A \subset \Z{p} \text{ with } |A| \le \alpha p \right\}.$$ 
\end{definition}

Thus, when $\bar\lambda = (1,-2)$, for example, we have $\ex\big( \Z{p},\bar\lambda,\alpha \big) = 3$ if $\alpha > 0$ is sufficiently small, and $\ex\big( \Z{p},\bar\lambda,\alpha \big) \le 1/\alpha$ for every $\alpha \le 1/2$. (So, in particular, $\ex\big( \Z{p},\bar\lambda,\alpha \big) \to 2$ as $\alpha \to 1/2$ and $p\to \infty$). This suggests that $\ex\big( \Z{p},\bar\lambda,\alpha \big)$ might exhibit some non-trivial behaviour between the two extremes. The main open question, of course,  is to understand the asymptotic behaviour of $\ex\big( \Z{p},\bar\lambda,\alpha \big)$ for all values of $\alpha \in [0,1]$.  

\begin{question}
Determine the value of  $\ex\big( \Z{p},\bar \lambda,\alpha \big)$ as $p \to \infty$ for all $\alpha \in [0,1]$.
\end{question}

	In this work, we have shown that Conjecture $\ref{con1}$ holds in the case where $\bar\lambda=(1,\lambda)$. Nonetheless, the transference principle still applies to an arbitrary $\bar\lambda=(\lambda_1,\lambda_2)$ and hence it suffices to construct large sets $A\subset \mathbb{T}$ (of measure $1/2-\epsilon$)  for which $\lambda_1(A)\cap\lambda_2(A)=\emptyset$. 
	
\begin{question}
\label{circle}
Let $\lambda_1, \lambda_2 \in \Z{}$ and suppose $A\subseteq \mathbb{T}$  is a Lebesgue measurable set such that $\lambda_1A\cap \lambda_2 A=\emptyset$. How large can $\mu(A)$ be? 
\end{question}

However, when dealing with higher order sumsets, for instance $\bar\lambda=(\lambda_1,\lambda_2,\lambda_3)$, it  looks likely that substantial new ideas will be required.

\begin{akn}
The work was inspired by a a series of conversations with Pablo Candela and David Saxton at the University of Cambridge. The author is very grateful to Ben Green, who pointed out the fact that one could remove the measure theory to obtain a quantitative version of the theorem. Last but not least, a special thanks to Robert Morris for his many helpful suggestions to improve the clarity and exposition of this paper. 
\end{akn}

\bibliographystyle{plain}
\bibliography{references} 
\end{document}